\documentclass[a4paper,11pt]{article}
\usepackage{amsmath,amsthm,amssymb,bbm,enumitem}
\usepackage[bookmarks=false]{hyperref}

\numberwithin{equation}{section}

{\theoremstyle{definition}\newtheorem{definition}{Definition}[section]

\newtheorem{remark}[definition]{Remark}

}

\newtheorem{proposition}[definition]{Proposition}
\newtheorem{lemma}[definition]{Lemma}
\newtheorem{theorem}[definition]{Theorem}
\newtheorem{corollary}[definition]{Corollary}

\newlist{enumlist}{enumerate}{1}
\setlist[enumlist]{labelindent=0cm,label=\arabic*$^\circ$,labelwidth=2.5ex,labelsep=0.5ex,leftmargin=3ex,align=left,topsep=0.5ex,itemsep=1ex,parsep=1ex}

\newlist{itemlist}{itemize}{1}
\setlist[itemlist]{labelindent=0cm,label=$\bullet$,labelwidth=2.5ex,labelsep=0.5ex,leftmargin=3ex,align=left,topsep=0.5ex,itemsep=1ex,parsep=1ex}

\newcommand{\eps}{\varepsilon}

\newcommand{\C}{\mathbb{C}}

\newcommand{\so}{\text{\it so}}

\newcommand{\ot}{\otimes}
\newcommand{\bA}{\mathbf{A}}
\newcommand{\bM}{\mathbf{M}}

\newcommand{\cM}{\mathcal{M}}

\newcommand{\ns}{\mathord{\text{\rm n}_{\text{\rm\tiny s}}}}

\newcommand{\cU}{\mathcal{U}}
\newcommand{\F}{\mathbb{F}}
\newcommand{\cMtil}{\widetilde{\mathcal{M}}}

\begin{document}

\begin{center}
{\boldmath\LARGE\bf On the optimal paving over MASAs \vspace{0.5ex}\\ in von Neumann algebras}

\bigskip

{\sc by Sorin Popa\footnote{Mathematics Department, UCLA, CA 90095-1555 (United States), popa@math.ucla.edu\\
Supported in part by NSF Grant DMS-1401718} and Stefaan Vaes\footnote{KU~Leuven, Department of Mathematics, Leuven (Belgium), stefaan.vaes@wis.kuleuven.be \\
    Supported by ERC Consolidator Grant 614195 from the European Research Council under the European Union's Seventh Framework Programme.}}
\end{center}

\begin{abstract}\noindent
We prove that if $A$ is a singular MASA in a II$_1$ factor $M$ and $\omega$ is a free ultrafilter,  then for any
$x\in M\ominus A$, with $\|x\|\leq 1$, and any $n\geq 2$,
there exists a partition of $1$ with projections $p_1, p_2, ..., p_n\in A^\omega$ (i.e. a {\it paving}) such that $\|\Sigma_{i=1}^n p_i x p_i\|\leq 2\sqrt{n-1}/n$, and give examples
where this is sharp. Some open problems on optimal pavings are discussed.
\end{abstract}

\section{Introduction}

A famous problem formulated by R.V. Kadison and I.M. Singer in 1959 asked whether
the diagonal MASA (maximal abelian $^*$-subalgebra) $\mathcal D$ of the algebra $\mathcal B(\ell^2\mathbb N)$, of all linear bounded operators on
the Hilbert space $\ell^2\mathbb N$, satisfies the {\it paving property}, requiring that for any contraction $x=x^*\in \mathcal B(\ell^2\mathbb N)$ with $0$ on the diagonal,
and any $\varepsilon > 0$, there exists a partition of $1$ with projections $p_1, ..., p_n \in \mathcal D$, such that $\|\sum_i p_i x p_i\|\leq \varepsilon$.
This problem has been settled in the affirmative by A. Marcus, D. Spielman and N. Srivastava in \cite{MSS13},
with an actual estimate $n \leq 12^4\eps^{-4}$ for the {\it paving size}, i.e., for the minimal number $n=n(x,\eps)$ of such projections.

In a recent paper \cite{PV14}, we considered a notion of paving
for an arbitrary MASA in a von Neumann algebra $A\subset M$,
that we called $\so$-{\it paving},
which requires that for any $x=x^*\in M$ and any $\varepsilon >0$, there exist $n\geq 1$, a net of partitions of $1$ with $n$ projections  $p_{1,i}, ..., p_{n,i}\in A$
and projections $q_i\in M$ such that $\|q_i(\Sigma_{k=1}^n p_{k,i}xp_{k,i} - a_i)q_i\|\leq \eps$, $\forall i$, and $q_i \rightarrow 1$ in the $\so$-topology.

This property
is in general weaker than the classic Kadison-Singer norm paving, but it coincides with it for the diagonal MASA
$\mathcal D\subset \mathcal B(\ell^2\mathbb N)$. We conjectured in \cite{PV14}
that any MASA $A\subset M$ satisfies
$\so$-paving. We used the results in \cite{MSS13} to check this conjecture
for all MASAs in type I von Neumann algebras, and all Cartan MASAs in amenable von Neumann algebras and in group measure space factors arising from
profinite actions, with the estimate $12^4\eps^{-4}$ for the $\so$-paving size derived from \cite{MSS13} as well.

We also showed  in \cite{PV14} that if $A$ is the range
of a normal conditional expectation, $E:M\rightarrow A$, and $\omega$ is a free ultrafilter on $\mathbb N$,
then $\so$-paving for $A\subset M$ is equivalent to the usual Kadison-Singer paving for the
ultrapower MASA $A^\omega\subset M^\omega$, with the norm paving size for $A^\omega\subset M^\omega$
coinciding with the $\so$-paving size for $A\subset M$. In the case $A$ is a singular MASA in a II$_1$ factor $M$,
norm-paving for the ultrapower inclusion $A^\omega \subset M^\omega$ has been established in \cite{P13},
with paving size $1250 \eps^{-3}$. This estimate
was improved   to $< 16\eps^{-2}+1$ in \cite{PV14}, while also shown to be $\geq \eps^{-2}$ for arbitrary MASAs in II$_1$ factors.

In this paper we prove that the paving size for singular MASAs in II$_1$ factors is in fact $< 4\eps^{-2}+1$,
and that for certain singular MASAs this is sharp.
More precisely, we prove that for any contraction $x\in M^\omega$ with $0$ expectation onto $A^\omega$, and for any $n\geq 2$, there
exists a partition of $1$ with $n$ projections $p_i \in A^\omega$ such that $\|\Sigma_{i=1}^n p_i x p_i \|\leq 2\sqrt{n-1}/n$. In fact,
given any finite set of contractions $F\subset M^\omega \ominus A^\omega$, we can find a partition $p_1, ..., p_n \in A^\omega$
that satisfies this estimate for all $x\in F$, so even the {\it multipaving size}  for singular MASAs is $< 4\eps^{-2}+1$.

To construct pavings satisfying this estimate, we first use Theorem 4.1(a) in \cite{P13} to  get
a unitary $u\in A^\omega$ with $u^n=1$, $\tau(u^k)=0$, $1\leq k \leq n-1$,  such that
any word with alternating letters from $\{u^k \mid 1\leq k \leq n-1\}$ and $F\cup F^*$ has trace $0$. This implies that for each
$x\in F$ the set $X = \{u^{i-1}xu^{-i+1} \mid i=1, 2, ..., n\}$ satisfies the conditions
$\tau(\Pi_{k=1}^m (x_{2k-1} x_{2k}^*))=0=\tau(\Pi_{k=1}^m (x_{2k-1}^* x_{2k}))$, for all $m$ and all $x_k \in X$ with $x_k \neq x_{k+1}$ for all $k$. We call {\it L-freeness} this property of a subset of a II$_1$ factor.
We then prove the general result, of independent interest,
that any L-free set of contractions $\{x_1,\ldots,x_n\}$ satisfies the norm estimate $\|\Sigma_{i=1}^n x_i \|\leq 2\sqrt{n-1}$. We do this
by first ``dilating'' $\{x_1, ..., x_n\}$ to an L-free set of unitaries $\{U_1, ..., U_n\}$ in a larger II$_1$ factor, for which we deduce
the Kesten-type estimate $\|\Sigma_{i=1}^n U_i\|
= 2\sqrt{n-1}$ from results in \cite{AO74}. This implies the inequality for the L-free contractions as well. By applying this to $\{u^{i-1}xu^{1-i} \mid i=1,\ldots,n\}$ and
taking into account that $\frac{1}{n}\Sigma_{i=1}^n u^{i-1}xu^{1-i}=\Sigma_{i=1}^n p_i x p_i$, where $p_1, ..., p_n$ are the minimal spectral projections of $u$,
we get $\|\Sigma_{i=1}^n p_i x p_i \|\leq 2\sqrt{n-1}/n$, $\forall x\in F$.

We also notice that if $M$ is a II$_1$ factor,  $A\subset M$ is a MASA and $v\in M$ a self-adjoint unitary of trace $0$
which is free with respect to $A$,
then $\|\Sigma_{i=1}^n p_i v p_i\| \geq 2\sqrt{n-1}/n$ for any partition of $1$ with projections in $A^\omega$, with equality if and only if $\tau(p_i)=1/n$, $\forall i$.
A concrete example is when $M=L(\mathbb Z * (\mathbb Z/2\mathbb Z))$, $A=L(\mathbb Z)$ (which is a singular MASA in $M$ by \cite{P81})
and $v=v^*\in L(\mathbb Z/2\mathbb Z)\subset M$ denotes the canonical generator.
This shows that the estimate $4\eps^{-2}+1$
for the paving size is in this case optimal.

The constant $2\sqrt{n-1}$ is known to coincide with the spectral radius of the $n$-regular tree, and with the first eigenvalue less than $n$ of  
$n$-regular Ramanujan graphs. Its occurence in this context leads us to   a more refined version of a conjecture formulated in \cite{PV14}, 
predicting that for any MASA $A\subset M$ which is range of a normal conditional expectation, 
any $n \geq 2$ and any contraction $x=x^*\in M$ with $0$ expectation onto $A$, the infimum 
$\eps(A\subset M; n, x)$ over all  
norms of pavings of $x$, $\| \Sigma_{i=1}^n p_i x p_i\| $, with $n$ projections $p_1, ..., p_n$ in $A^\omega$, $\Sigma_i p_i=1$, is bounded above 
by $2\sqrt{n-1}/n$, and that in fact $\sup \{\eps(A\subset M; n, x) \mid x=x^*\in M \ominus A, \|x\|\leq 1\}=2\sqrt{n-1}/n$. Such an optimal 
estimate would be particularly interesting to establish for the diagonal MASA $\mathcal D \subset \mathcal B(\ell^2\mathbb Z)$. 

\section{Preliminaries}

A well known result of H. Kesten in \cite{Ke58} shows that if $\mathbb F_k$ denotes the free group with $k$ generators $h_1, ..., h_k$, and $\lambda$ is the left
regular representation of $\mathbb F_k$ on $\ell^2\mathbb F_k$, then
the norm of the {\it Laplacian operator} $L=\Sigma_{i=1}^k (\lambda(h_i) + \lambda (h_i^{-1}))$ is equal to $2\sqrt{2k-1}$.  It was also shown in \cite{Ke58}
that, conversely, if $k$ elements $h_1, ..., h_k$ in a group $\Gamma$ satisfy $\|\Sigma_{i=1}^k \lambda(h_i) + \lambda (h_i^{-1})\|=2\sqrt{2k-1}$, then $h_1, ..., h_k$ are freely
independent, generating a copy of $\mathbb F_k$ inside $\Gamma$. The calculation of the norm of $L$ in \cite{Ke58} uses the formalism of
random walks on groups, but it really amounts to calculating the higher moments $\tau(L^{2n})$
and using the formula $\|L\|=\lim_m(\tau(L^{2m}))^{1/2m}$, where $\tau$ denotes the canonical (normal faithful) tracial state on the group von Neumann algebra $L(\mathbb F_k)$.

Kesten's result implies that whenever $u_1, ..., u_k$ are
freely independent Haar unitaries in a type II$_1$ factor $M$
(i.e., $u_1, ..., u_k$ generate a copy of $L(\mathbb F_k)$ inside $M$), then one has $\| \Sigma_{i=1}^k u_i + u_i^*\|= 2 \sqrt{2k-1}$. In particular, if $M$ is the free group factor
$L(\mathbb F_k)$ and $u_i=\lambda(h_i)$, where $h_1, ..., h_k\in \mathbb F_k$ as above, then $\| \Sigma_{i=1}^k \alpha_iu_i + \overline{\alpha_i}u_i^*\|= 2 \sqrt{2k-1}$,
for any scalars $\alpha_i \in \mathbb C$ with $|\alpha_i|=1$.

Estimates of norms of linear combinations of elements satisfying more general free independence relations
in group II$_1$ factors $L(\Gamma)$ have later been obtained in \cite{L73}, \cite{B74}, \cite{AO74}\footnote{See also the more ``rough'' norm estimates for elements in $L(\mathbb F_n)$ obtained by R. Powers in 1967 in relation to another problem of Kadison,
but published several years later in \cite{Po75}, and which motivated in part the work in \cite{AO74}.}.
These estimates involve elements in $L(\Gamma)$
(viewed as convolvers on $\ell^2\Gamma$)
that are supported on a subset $\{g_1,\ldots,g_n\} \subset \Gamma$ satisfying the following
weaker freeness condition, introduced in \cite{L73}: whenever $k \geq 1$ and $i_s \neq j_s$, $j_s \neq i_{s+1}$ for all $s$, we have that
$$g_{i_1} g_{j_1}^{-1} \cdots g_{i_k} g_{j_k}^{-1} \neq e \; .$$
In \cite{B74} and \cite{AO74}, this is called the Leinert property and it is proved
to be equivalent with $\{g_1^{-1} g_2,\ldots,g_1^{-1} g_n\}$ freely generating a copy of $\F_{n-1}$.
The most general calculation of norms of elements $x=\Sigma_i c_i \lambda(g_i) \in L(\Gamma)$,
supported on a Leinert set $\{g_i\}_i$, with arbitrary coefficients $c_i \in \mathbb C$,
was obtained by Akemann and Ostrand in \cite{AO74}.
The calculation shows in particular that if $\{g_1, \ldots, g_n\}$ satisfies Leinert's freeness condition then
$\|\Sigma_{i=1}^n \lambda(g_i)\|=2\sqrt{n-1}$.
Since $h_1, ..., h_k\in \Gamma$ freely independent implies $\{h_i, h_i^{-1} \mid 1\leq i \leq k\}$ is a Leinert set, the result in \cite{AO74} does recover
Kesten's theorem as well. Like in \cite{Ke58}, the norm of an element of the form $L=\Sigma_{i=1}^n c_i \lambda(g_i)$ in \cite{AO74}
is calculated by evaluating $\lim_n \tau((L^*L)^n)^{1/2n}$ (by
computing the generating function of the moments of $L^*L$).

An argument similar to \cite{Ke58} was used in \cite{Le96}
to prove that, conversely, if some elements $g_1, ..., g_n$ in a group $\Gamma$ satisfy $\|\Sigma_{i=1}^n \lambda(g_i)\|=2\sqrt{n-1}$, then $g_1, ..., g_n$
is a Leinert set. On the other hand, note that if $g_1, ..., g_n$ are $n$ arbitrary elements in an arbitrary group $\Gamma$ and we denote $L=\Sigma_{i=1}^n\lambda(g_i)$
the corresponding Laplacian, then the $n$'th moment $\tau((L^*L)^n)$ is bounded from below by the $n$'th moment of the Laplacian obtained by taking
$g_i$ to be the generators of $\mathbb F_n$. Thus, we always have $\|\Sigma_{i=1}^n\lambda(g_i)\|\geq 2\sqrt{n-1}$. More generally, if
$v_1, ..., v_n$ are unitaries in a von Neumann algebra $M$ with normal faithful trace state $\tau$,
such that any word $v_{i_1}v_{j_1}^*v_{i_2}v_{j_2}^*.... v_{i_m}v^*_{j_m}$, $\forall m\geq 1$,
$\forall 1\leq i_k, j_k \leq n$, has trace with non-negative real part, then
$\|\Sigma_{i=1}^n v_i \|\geq 2\sqrt{n-1}$. In particular, for any unitaries $u_1, ..., u_n \in M$ one has $\|\Sigma_{i=1}^n u_i \otimes \overline{u_i}\|\geq 2\sqrt{n-1}$.

For convenience, we state below some norm calculations from \cite{AO74}, formulated in the form that will be used in the sequel:

\begin{proposition}[\cite{AO74}] If $v_1, v_2, ..., v_{n-1}\in M$ are freely independent Haar unitaries, then
\begin{equation}\label{eq.211}
\|1 + \Sigma_{i=1}^{n-1} v_i \| = 2 \sqrt{n-1}.
\end{equation}
Also, if $\alpha_0, ..., \alpha_{n-1}\in \mathbb C$, $\Sigma_i |\alpha_i|^2=1$, then
\begin{equation}\label{eq.212}
\|\alpha_01+\Sigma_{i=1}^{n-1}\alpha_i v_i\|\leq 2\sqrt{1-1/n}.
\end{equation}
\end{proposition}

Note that \eqref{eq.211} above shows in particular that if $p, q\in M$ are projections with $\tau(p)=1/2$ and $\tau(q)=1/n$, for some $n \geq 3$,  and
they are freely independent, then $\|qpq\|=1/2+\sqrt{n-1}/n$. Indeed, any two such projections can be thought of as embedded into $L(\mathbb F_2)$
with $p$ and $q$ lying in the MASAs of the two generators, $p\in A_1$, respectively $q\in A_2$. Denote $v=2p-1$. Let
$q_1=q, q_2, ..., q_n \in A_2$ be mutually
orthogonal projections of trace $1/n$ and denote $u=\Sigma_{j=1}^{n}\lambda^{j-1} q_j$, where $\lambda=2\exp(2\pi i/n)$.
It is then easy to see that the elements
$v_k=vu^kvu^{-k}$, $k=1, 2, ..., n-1$ are  freely independent Haar unitaries. By \eqref{eq.211} we thus have
$\|\Sigma_{k=0}^{n-1} u^kvu^{-k}\|=\|1+\Sigma_{k=1}^{n-1}vu^kvu^{-k}\|=2\sqrt{n-1}.$
But $\Sigma_{k=0}^{n-1} u^kvu^{-k}= n (\Sigma_{j=1}^n q_j v q_j)$,
implying that

$$\|qvq\|=\|q(2p-1)q\|=2\sqrt{n-1}/n=2\sqrt{\tau(q)(1-\tau(q))}$$
or equivalently

$$
\|qpq\|=1/2+\sqrt{n-1}/n=\tau(p)+\sqrt{\tau(q)(1-\tau(q))}.
$$

The computation of the norm of the product of freely independent projections $q, p$ of arbitrary trace in $M$
(in fact, of the whole spectral distribution of $qpq$) was obtained by Voiculescu in \cite{Vo86}, as one of the first applications of his
multiplicative free convolution (which later became a powerful tool in free probability). We recall here these norm estimates, which in particular
show that the first of the above norm calculations holds true for projections $q$ of arbitrary trace (see also \cite{ABH87} for the case
$\tau(q)=1/n$, $\tau(p)=1/m$, for integers $n \geq m\geq 2$):

\begin{proposition}[\cite{Vo86}] \label{prop.22}
If $p, q\in M$ are freely independent projections with $\tau(q)\leq \tau(p)\leq 1/2$, then
\begin{equation}\label{eq.221}
\|qpq\| = \tau(p) + \tau(q)- 2\tau(p)\tau(q) + 2\sqrt{\tau(p) \, \tau(1-p) \, \tau(q) \, \tau(1-q)}.
\end{equation}
If in addition $\tau(p)=1/2$ and we denote $v=2p-1$, then
\begin{equation}\label{eq.222}
\|qvq\|= 2\sqrt{\tau(q) \, \tau(1-q)}.
\end{equation}
\end{proposition}

\section{$L$-free sets of contractions and their dilation}

Recall from \cite{P13} that two selfadjoint sets $X, Y \subset M\ominus \mathbb C1$ of a tracial von Neumann algebra $M$
are called {\it freely independent sets\footnote{We specifically consider this condition for subsets $X,Y \subset M \ominus \C 1$, not to be confused with the freeness of the von Neumann algebras generated by $X$ and $Y$.}} if the trace of any word with letters alternating
from $X$ and $Y$ is equal to $0$. Also, a subalgebra $B\subset M$ is called {\it freely independent of a set} $X$, if $X$ and $B\ominus \mathbb C1$ are freely
independent as sets.  Several results were obtained in \cite{P13} about constructing a ``large subalgebra'' $B$ inside a given subalgebra $Q\subset M$ that is freely independent of a given countable set $X$. Motivated by a condition appearing in one such result, namely \cite[Theorem 4.1]{P13}, and by a terminology used in \cite{AO74},
we consider in this paper the following free independence condition for arbitrary elements in tracial algebras:

\begin{definition}
Let $(M,\tau)$ be a von Neumann algebra with a normal faithful tracial state. A subset $X \subset M$ is called \emph{L-free\footnote{Note that this notion is not the same as (and should not be confused with) the notion of L-sets used in \cite{Pi92}.}} if
$$\tau(x_1 x_2^* \cdots x_{2k-1} x_{2k}^*) = 0 \quad\text{and}\quad \tau(x_1^* x_2 \cdots x_{2k-1}^* x_{2k}) = 0 \; ,$$
whenever $k \geq 1$, $x_1,\ldots,x_{2k} \in X$ and $x_i \neq x_{i+1}$ for all $i=1,\ldots,2k-1$.
\end{definition}

Note that if the subset $X$ in the above definition is taken to be contained in the set of canonical unitaries
$\{u_g \mid g \in \Gamma \}$ of a group von Neumann algebra $M=L(\Gamma)$, i.e.\ $X = \{u_g \mid g \in F\}$ for some subset $F \subset \Gamma$,
then L-freeness of $X$ amounts to $F$ being a Leinert set. But the key example of an L-free set that is important for us  here
occurs from a diffuse algebra $B$ that is free independent from a set $Y=Y^*\subset M\ominus \mathbb C1$:  given any $y_1,\ldots,y_n \in Y$ and any unitary element
$u\in \mathcal U(B)$ with $\tau(u^k)=0$, $1\leq k \leq n-1$, the set $\{u^{k-1} y_k u^{-k+1} \mid 1 \leq k \leq n\}$ is L-free.

Note that we do need to impose both conditions on the traces being zero in Definition 3.1, because we cannot deduce $\tau(x_1^* x_2 x_3^* x_1) = 0$ from $\tau(y_1 y_2^* y_3 y_4^*) = 0$ for all $y_i \in X$ with $y_1 \neq y_2$, $y_2 \neq y_3$, $y_3 \neq y_4$. However, if $X \subset \cU(M)$ consists of unitaries, then only one set of conditions is sufficient. We in fact have:

\begin{lemma} Let $X=\{u_1,\ldots,u_n\}\subset \cU(M)$.  Then the following conditions are equivalent

$(a)$ $X$ is an L-free set.

$(b)$ $\tau(u_{i_1} u_{j_1}^* \cdots u_{i_k} u_{j_k}^*) = 0$ whenever $k \geq 1$ and $i_s \neq j_s$, $j_s \neq i_{s+1}$ for all $s$.

$(c)$ $u_1^* u_2,\ldots,u_1^* u_n$ are free generators of a copy of $L(\F_{n-1})$.
\end{lemma}
\begin{proof} This  is a trivial verification.
\end{proof}

\begin{corollary}\label{cor.33}
If $\{u_1, ..., u_n\}$ is an L-free set of unitaries in $\cU(M)$, then
$\| \Sigma_{i=1}^n u_i \| = 2 \sqrt{n-1}$. Moreover, if $\alpha_1, ..., \alpha_n\in \mathbb C$ with $\Sigma_{i=1}^n |\alpha_i|^2\leq 1$, then
$$
\Bigl\| \sum_{i=1}^n \alpha_i u_i \Bigr\| \leq 2\sqrt{1-1/n}.
$$
\end{corollary}
\begin{proof} Since $\|\Sigma_{i=1}^n \alpha_i u_i\|=\|\alpha_1 1 + \Sigma_{i=2}^n \alpha_i u_1^*u_i\|$, the statement follows by applying \eqref{eq.212} to
the freely independent Haar unitaries $v_j=u_1^*u_j$, $2\leq j \leq n$.
\end{proof}

\begin{proposition}\label{prop.34}
Let $M$ be a finite von Neumann algebra with a faithful tracial state $\tau$. If $\{x_1,\ldots,x_n\} \subset M$ is an L-free set with $\|x_i\| \leq 1$ for all $i$, then
there exists a tracial von Neumann algebra $(\cM,\tau)$, a trace preserving unital embedding $M \subset \cM$ and an L-free set of unitaries $\{U_1,\ldots,U_n\} \subset \cU(\cMtil)$ with $\cMtil = M_{n+1}(\C) \ot \cM$ so that, denoting by $(e_{ij})_{i,j=0,\ldots,n}$ the matrix units of $M_{n+1}(\C)$, we have $e_{00} U_i e_{00} = x_i$ for all $i$.
\end{proposition}
\begin{proof}
Define $\cM = M * L(\F_{n(n-1)})$ and denote by $u_{i,j}$, $i \neq j$, free generators of $L(\F_{n(n-1)})$. For every $i \in \{1,\ldots,n\}$, define
$$c_i = \sqrt{1-x_i x_i^*} \quad\text{and}\quad d_i = - \sqrt{1-x_i^* x_i} \; .$$
Put $\cMtil = M_{n+1}(\C) \ot \cM$ and define the unitary elements $U_i \in \cU(\cMtil)$ given by
$$U_i = (e_{00} \ot x_i) + (e_{ii} \ot x_i^*) + (e_{0i} \ot c_i) + (e_{i0} \ot d_i) + \sum_{j \neq i} (e_{jj} \ot u_{i,j}) \; .$$
Note that $U_i$ is the direct sum of the unitary
$$\begin{pmatrix} x_i & c_i \\ d_i & x_i^* \end{pmatrix} \;\;\text{in positions $0$ and $i$, and the unitary}\;\; \bigoplus_{j \neq i} u_{i,j} \;\;\text{in the positions $j \neq i$.}$$
By construction, we have that $e_{00} U_i e_{00} = x_ie_{00} $. So, it remains to prove that $\{U_1,\ldots,U_n\}$ is L-free.

Take $k \geq 1$ and indices $i_s,j_s$ such that $i_s \neq j_s$, $j_s \neq i_{s+1}$ for all $s$. We must prove that
\begin{equation}\label{eq.goal}
\tau (U_{i_1} U_{j_1}^* \cdots U_{i_k} U_{j_k}^*) = 0 \; .
\end{equation}
Consider $V := U_{i_1} U_{j_1}^* \cdots U_{i_k} U_{j_k}^*$ as a matrix with entries in $\cM$. Every entry of this matrix is a sum of ``words'' with letters
$$\{ x_i,x_i^*,c_i,d_i \mid i = 1,\ldots,n\} \cup \{u_{i,j} , u_{i,j}^* \mid i \neq j \} \; .$$
We prove that every word that appears in a diagonal entry $V_{ii}$ of $V$ has zero trace. The following types of words appear.

\begin{enumlist}
\item Words without any of the letters $u_{a,b}$ or $u_{a,b}^*$. These words only appear as follows:
\begin{itemlist}
\item in the entry $V_{00}$ as $x_{i_1} x_{j_1}^* \cdots x_{i_k} x_{j_k}^*$, which has zero trace;
\item if $i_1 = j_k = i$, in the entry $V_{ii}$ as $w = d_i x_{j_1}^* x_{i_2} x_{j_2}^* \cdots x_{i_{k-1}} x_{j_{k-1}}^* x_{i_k} d_i^*$. Then we have
\begin{align*}
\tau(w) &= \tau(x_{j_1}^* x_{i_2} \cdots x_{j_{k-1}}^* x_{i_k} \, d_i^* d_i) \\
&= \tau(x_{j_1}^* x_{i_2} \cdots x_{j_{k-1}}^* x_{i_k}) - \tau(x_{j_1}^* x_{i_2} \cdots x_{j_{k-1}}^* x_{i_k} \, x_i^* x_i) \\
&= 0 - \tau(x_{i_1} x_{j_1}^* \cdots x_{i_k} x_{j_k}^*) = 0 \; ,
\end{align*}
because $i = i_1$ and $i = j_k$.
\end{itemlist}

\item Words with exactly one letter of the type $u_{a,b}$ or $u_{a,b}^*$. These words have zero trace because $\tau(M u_{a,b} M) = \{0\}$.

\item Words $w$ with two or more letters of the type $u_{a,b}$ or $u_{a,b}^*$. Consider two consecutive such letters in $w$, i.e.\ a subword of $w$ of the form
$$u_{i,j}^{\eps} \; w_0 \; u_{i',j'}^{\eps'}$$
with $\eps,\eps' = \pm 1$ and where $w_0$ is a word with letters from $\{ x_i,x_i^*,c_i,d_i \mid i = 1,\ldots,n\}$. We distinguish three cases.
\begin{itemlist}
\item $(\eps',i',j') \neq (-\eps,i,j)$.
\item $u_{i,j} \; w_0 \; u_{i,j}^*$.
\item $u_{i,j}^* \; w_0 \; u_{i,j}$.
\end{itemlist}
To prove that $\tau(w) = 0$, it suffices to prove that in the last two cases, we have that $\tau(w_0)=0$.

A subword of the form $u_{i,j} \; w_0 \; u_{i,j}^*$ can only arise from the $jj$-entry of
$$U_{i_s} U_{j_s}^* \cdots U_{i_t} U_{j_t}^* \quad\text{with}\;\; i_s = j_t = i \; , \; j_s = i_t = j$$
(and thus, $t \geq s+2$). In that case,
$$w_0 = c_j^* \, x_{i_{s+1}} x_{j_{s+1}}^* \cdots x_{i_{t-1}} x_{j_{t-1}}^* \, c_j \; .$$
Thus,
\begin{align*}
\tau(w_0) &= \tau(x_{i_{s+1}} x_{j_{s+1}}^* \cdots x_{i_{t-1}} x_{j_{t-1}}^* \, c_j c_j^*) \\
&= \tau(x_{i_{s+1}} x_{j_{s+1}}^* \cdots x_{i_{t-1}} x_{j_{t-1}}^*) - \tau(x_{i_{s+1}} x_{j_{s+1}}^* \cdots x_{i_{t-1}} x_{j_{t-1}}^* \, x_j x_j^*) \\
&= 0 - \tau(x_{j_s} x_{i_{s+1}}^* \cdots x_{j_{t-1}}^* x_{i_t}) = 0 \; ,
\end{align*}
because $j = j_s$ and $j = i_t$.

Finally, a subword of the form $u_{i,j}^* \; w_0 \; u_{i,j}$ can only arise from the $jj$-entry of
$$U_{j_{s-1}}^* U_{i_s} \cdots U_{j_{t-1}}^* U_{i_t} \quad\text{with}\;\; j_{s-1} = i_t = i \; , \; i_s = j_{t-1} = j$$
(and thus, $t \geq s+2$). In that case,
$$w_0 = d_j \, x_{j_s}^* x_{i_{s+1}} \cdots x_{j_{t-2}}^* x_{i_{t-1}} \, d_j^* \; .$$
As above, it follows that $\tau(w_0) = 0$.
\end{enumlist}

So, we have proved that every word that appears in a diagonal entry $V_{ii}$ of $V$ has trace zero. Then also $\tau(V) = 0$ and it follows that $\{U_1,\ldots,U_n\}$ is an L-free set of unitaries.
\end{proof}

\begin{corollary}
Let  $(M, \tau)$ be a finite von Neumann algebra with a faithful normal tracial state. If $\{x_1,\ldots,x_n\} \subset M$ is L-free with $\|x_i\| \leq 1$ for all $i$, then
$$\Bigl\| \sum_{i=1}^n x_i \Bigr\| \leq 2 \sqrt{n-1} \; .$$
More generally, given any complex scalars $\alpha_1, ..., \alpha_n$ with $\Sigma_{i=1}^n |\alpha_i|^2\leq 1$, we have
$$\Bigl\| \sum_{i=1}^n \alpha_i x_i \Bigr\| \leq 2 \sqrt{1-1/n} \; .$$
\end{corollary}
\begin{proof}
Assuming $n \geq 2$, with the notations from Proposition \ref{prop.34} and by using Corollary \ref{cor.33},
we have $\Bigl\| \sum_{i=1}^n \alpha_i U_i \Bigr\| \leq 2\sqrt{1-1/n}$.
Reducing with the projection $e_{00}$, it follows that
$$\Bigl\| \sum_{i=1}^n \alpha_i x_i \Bigr\| \leq 2 \sqrt{1-1/n} \; .$$
\end{proof}

\section{Applications to paving problems}

Like in \cite{P13}, \cite{PV14}, if  $\mathcal A \subset \mathcal M$ is a MASA in a von Neumann algebra and $x\in \mathcal M$, then we denote by
$\text{\rm n}(\mathcal A\subset \mathcal M; x,\eps)$ the smallest $n$ for which
there exist projections $p_1,\ldots,p_n \in \mathcal A$ and
$a \in \mathcal A$ such that $\|a\| \leq \|x\|$, $\sum_{i=1}^n p_i = 1$ and $\Bigl\| \sum_{i=1}^n p_i x p_i - a \Bigr\| \leq \eps \|x\|$ (with the convention
that $\text{\rm n}(\mathcal A\subset \mathcal M; x,\eps)=\infty$ if no such finite partition exists), and call it the {\it paving size} of $x$.

Recall also from \cite{D54} that a MASA $\mathcal A$ in a von Neumann algebra $\mathcal M$ is called {\it singular}, if the only unitary elements in $\mathcal M$
that normalize $\mathcal A$ are the unitaries in $\mathcal A$.

\begin{theorem}\label{thm.41}
Let $A_n \subset M_n$ be a sequence of singular MASAs in finite von Neumann algebras
and $\omega$ a free ultrafilter on $\mathbb N$.
Denote $\bM = \prod_\omega M_n$ and $\bA = \prod_\omega A_n$.
Given any countable  set of contractions
$X \subset \bM \ominus \bA$ and any integer $n\geq 2$,
there exists a partition of $1$ with projections $p_1, ..., p_n \in \bA$ such that
$$\Bigl\| \sum_{j=1}^n p_j x p_j \Bigr\| \leq 2\sqrt{n-1}/n, \quad\text{for all}\;\; x \in X \; .$$
In particular, the paving size of $\bA \subset \bM$,
$$\text{\rm n}(\bA\subset \bM; \eps) \mathbin{\overset{\text{\rm\scriptsize def}}{=}} \sup \{\text{\rm n}(\bA\subset \bM; x, \eps)
\mid x=x^*\in \bM \ominus \bA\} \; ,$$
is less than  $4\eps^{-2} +1$, for any $\eps>0$.
\end{theorem}
\begin{proof} By Theorem 4.1(a) in \cite{P13}, there exists a diffuse abelian von Neumann subalgebra
$A_0 \subset \bA$ such that for any $k\geq 1$, any word with alternating letters $x=x_0 \Pi_{i=1}^k (v_ix_i)$ with $x_i\in X$, $1\leq i \leq k-1$,
$x_0, x_k \in X\cup\{1\}$,
$v_i\in A_0\ominus \mathbb C1$, has trace equal to $0$.

This implies that if $p_1, ..., p_n\in \bA$ are projections of trace $1/n$ summing up to $1$ and
we denote $u=\Sigma_{j=1}^n \lambda^{j-1}p_j$,
where $\lambda=\exp (2\pi i/n)$, then for any $x\in X$ the set $\{ u^{i-1}xu^{-i+1} \mid i=1, 2, ..., n\}$  is L-free. Since
$\frac{1}{n}\Sigma_{i=1}^n u^{i-1}xu^{1-i}=\Sigma_{i=1}^n p_i x p_i$, where $p_1, ..., p_n$ are the minimal spectral projections of $u$,
by Proposition \ref{prop.34} it follows that  for all $x\in X$ we have
$$
\|\Sigma_{i=1}^n p_i x p_i \| = \frac{1}{n} \, \|\Sigma_{i=1}^n u^{i-1}xu^{-i+1}\| \leq 2\sqrt{n-1}/n.$$

To derive the last part, let $\eps>0$ and denote by $n$ the integer with the property that $2n^{-1/2}\leq \eps< 2(n-1)^{-1/2}$. If $x\in \bM\ominus \bA$, $\|x\|\leq 1$,
and $p_1, ..., p_n\in \bA$ are mutually orthogonal projections of trace $1/n$ that satisfy the free independence relation with $X=\{x\}$ as above, then $n< 4\eps^{-2}+1$
and we have
$$
\|\Sigma_{i=1}^n p_i x p_i \| \leq 2\sqrt{n-1}/n \leq \eps, $$
showing that $\text{\rm n}(\bA\subset \bM; x,\eps)< 4\eps^{-2}+1.$
\end{proof}

\begin{remark} The above result suggests that an alternative
way of measuring the {\so}-paving size over a MASA
in a von Neumann algebra $A\subset M$ admitting a normal conditional expectation,
is by considering the quantity

$$
\eps(A\subset M; n) \mathbin{\overset{\text{\rm\scriptsize def}}{=}} \sup_{x\in (M_h^\omega \ominus A^\omega)_1}
(\inf \{\| \Sigma_{i=1}^n p_i x p_i\| \mid p_i \in \mathcal P(A^\omega), \Sigma_i p_i=1\}).$$

With this notation, the above theorem
shows that for a singular MASA in a II$_1$ factor $A\subset M$,
one has $\eps(A\subset M; n)\leq 2\sqrt{n-1}/n$, $\forall n\geq 2$, a formulation that's slightly more precise than the estimate
$\ns(A\subset M; \eps) = \text{\rm n}(A^\omega \subset M^\omega; \eps) <  4\eps^{-2}+1$. Also, the conjecture (2.8.2$^\circ$ in \cite{PV14}) about
the {\so}-paving size can this way be made more precise, by asking whether $\eps(A\subset M; n) \leq 2\sqrt{n-1}/n$, $\forall n$, for any
MASA with a normal conditional expectation $A\subset M$. It seems particularly interesting to study
this question in the classical Kadison-Singer case of the diagonal MASA $\mathcal D \subset  \mathcal B=\mathcal B(\ell^2\mathbb N)$,
and more generally for Cartan MASAs $A\subset M$. So far,  the solution
to the Kadison-Singer paving problem in \cite{MSS13} shows that $\eps(\mathcal D\subset \mathcal B; n)\leq 12 n^{-1/4}$.

Also, while by \cite{CEKP07} one has $\text{\rm n}(\mathcal D \subset \mathcal B; \eps) \geq \eps^{-2}$
and by \cite{PV14} one has $\ns (A\subset M; \eps) = \text{\rm n}(A^\omega \subset M^\omega; \eps) \geq \eps^{-2}$,
for any MASA in a II$_1$ factor $A\subset M$, it  would be interesting to decide whether  $\eps(\mathcal D \subset \mathcal B; n)$ and
$\eps(A\subset M; n)$ are in fact bounded from below by $2\sqrt{n-1}/n$, $\forall n$.

For a singular MASA in a II$_1$ factor, $A\subset M$, combining \ref{thm.41} with such a lower bound would show that $\eps(A\subset M; n)=2\sqrt{n-1}/n$, $\forall n$.
While we could not prove this general fact, let us note here that for certain singular MASAs this equality holds indeed.
\end{remark}

\begin{proposition}
$1^\circ$ Let $M$ be a $\text{\rm II}_1$ factor and $A\subset M$ a MASA. Assume $v\in M$ is a unitary element with $\tau(v)=0$
such that $A$ is freely independent of the set $\{v, v^*\}$ $($i.e., any alternating word in $A\ominus \mathbb C1$ and
$\{v, v^*\}$ has trace $0)$. Then for any partition of $1$ with projections $p_1, ..., p_n \in A^\omega$ we have $\|\Sigma_{i=1}^n p_i v p_i \|\geq 2\sqrt{n-1}/n$,
with equality iff all $p_i$ have trace $1/n$. Also, $\eps(A\subset M; n)\geq 2\sqrt{n-1}/n$, $\forall n$.

$2^\circ$ If $M=L(\mathbb Z * (\mathbb Z/2\mathbb Z))$, $A=L(\mathbb Z)$ and $v=v^*$ denotes the canonical
generator of $L(\mathbb Z/2\mathbb Z)$, then $\eps(A\subset M; v, n)=\eps(A\subset M; n)=2\sqrt{n-1}$, $\forall n$.
\end{proposition}
\begin{proof} The free independence assumption in $1^\circ$ implies that  $A^\omega\ominus \mathbb C$ and $\{v, v^*\}$ are
freely independent sets as well. This in turn implies that for each $i$, the projections $p_i$ and $vp_iv^*$ are freely independent,
and so by Proposition \ref{prop.22} one has $\|p_ivp_i\|=\|p_ivp_iv^*\|=
2\sqrt{\tau(p_i)(1-\tau(p_i))}$.  Thus, if  one of the projections $p_i$ has trace $\tau(p_i)> 1/n$, then $\|\Sigma_j p_jvp_j\|\geq \|p_i v p_i\|> 2\sqrt{n-1}/n$,
while if $\tau(p_i)=1/n$, $\forall i$, then $\|\Sigma_j p_jvp_j\|= 2\sqrt{n-1}/n$.

By applying $1^\circ$ to part $2^\circ$, then using \ref{thm.41} and the fact that $A=L(\mathbb Z)$ is singular in $M=L(\mathbb Z * (\mathbb Z/2\mathbb Z))$
(cf.\ \cite{P81}), proves the last part of the statement.
\end{proof}

\end{document}